\documentclass[a4paper, reqno]{amsart}

\usepackage{amssymb}
\usepackage{amsthm}
\usepackage{amsfonts}
\usepackage{amsmath}
\usepackage{mathrsfs}
\usepackage[colorlinks=true,citecolor=black,linkcolor=black]{hyperref}
\usepackage{graphicx}
\usepackage{bbm}
\usepackage{mathtools}   
\usepackage{enumitem}

\newtheorem{thm}{Theorem}

\newtheorem{lem}[thm]{Lemma}

\newtheorem{cond}{Condition}

\theoremstyle{definition}

\DeclareMathAlphabet{\mathsc}{OT1}{cmr}{m}{sc}

\newcommand\cB{{\mathscr B}}
\newcommand\cC{{\mathscr C}}

\newcommand\cJ{{\mathcal J}}
\newcommand\cK{{\mathscr K}}
\newcommand\cL{{\mathscr L}}

\newcommand\cM{{\mathcal M}}

\newcommand\cT{{\mathcal T}}

\newcommand\bC{{\mathbb C}}

\newcommand\bN{{\mathbb N}}

\newcommand\bR{{\mathbb R}}

\newcommand{\norm}[1]{\left\lVert{#1}\right\rVert}
\newcommand{\dr}[1]{\operatorname{d}_{r}({#1})}
\newcommand{\de}[1]{\operatorname{d}_{E}({#1})}
\newcommand{\db}[1]{\operatorname{d}_{B}({#1})}
\newcommand{\dep}[1]{\operatorname{d}'({#1})}

\mathtoolsset{centercolon}   

\begin{document}
\author{Oliver Butterley}
\address{Oliver Butterley\\ Faculty of Mathematics, Universit\"at Wien,
Nordbergstra{\ss}e 15,
1090 Wien, Austria}
\email{oliver.butterley@univie.ac.at}
\author{Carlangelo Liverani}
\address{Carlangelo Liverani\\
Dipartimento di Matematica\\
II Universit\`{a} di Roma (Tor Vergata)\\
Via della Ricerca Scientifica, 00133 Roma, Italy.}
\email{{\tt liverani@mat.uniroma2.it}}
\title{Robustly invariant sets in fibre contracting bundle flows}
\thanks{The authors are grateful for the constructive comments of the anonymous referees which, after some work, led to a substantial improvement and generalisation of the argument. Also, they acknowledge the support of the  ERC Advanced Grant MALADY (246953).}
\bibliographystyle{abbrv}
\maketitle
\thispagestyle{empty}
\begin{abstract}
We provide abstract conditions which imply the existence of a robustly invariant neighbourhood of a global section of a fibre bundle flow. We then apply such a result to the  bundle flow generated by an Anosov flow when the fibre is the space of jets (which are described by local manifolds). As a consequence we obtain sets of manifolds (e.g. approximations of stable manifolds) that are left invariant, {\bf for all} negative times, by the flow and its small perturbations. Finally, we show that the latter result can be used to easily fix a mistake recently uncovered in the paper {\em Smooth Anosov flows: correlation spectra and stability},  \cite{BuL}, by the present authors.
\end{abstract}

\section{Introduction}
The standard definition of Anosov map or flow states that there must be an invariant splitting that is eventually contracted or expanded \cite{KH}. It is well known that one can change the Riemannian metric so as to insure that the hyperbolicity (expansion and contraction) takes place for {\bf all times} and not just after a fixed time (e.g. see \cite{mather}). Unfortunately, the new metric is usually constructed via the invariant splitting and hence has its same regularity (typically only H\"older). Also, if one considers a small perturbation of the flow or the map one has strict hyperbolicity with respect to a different metric. On the contrary, in many recent applications (e.g. the perturbation theory developed in \cite{GL} and applied to flows in \cite{BuL}) it would be very convenient to have strict hyperbolicity for an open set of systems with respect to a fixed (possibly smooth) distance. 

In fact, as typically one establishes hyperbolicity by finding a family of strictly invariant cones, it would be very useful to be able to specify a family of invariant cones   that is invariant also for small perturbations of the map or flow and has a uniform contraction and expansion for all the vectors in the cone (also for the perturbed dynamics). Moreover, as for many applications one would like to consider sets of approximate unstable manifolds and given that such sets can be conveniently specified in the language of jets, it is tempting to state the results in terms of fibre bundles, so that a manifold of different situations would be treated all at once. 

It turns out that it is fairly easy to carry out the above program in the case of discrete time. The reason being that the standard strategy provides, at time one, a fixed amount of contraction or expansion and hence a small perturbation (needed to make the metric smooth or to consider near by dynamics) will conserve such a property. Completely different is the situation for flows. The problem being the small times. For such times the amount of hyperbolicity is very small and any perturbation may destroy it. It is then not obvious if the above program can be carried out for flows or not. Here we show that it can be done. This is a consequence of the general considerations put forward in section \ref{sec:theory}. We then show in section \ref{sec:application} how to apply the abstract theory to the case of Anosov flows.

The problem of small times is rather ubiquitous when attempting to study fine properties of flows. In particular, in the theory developed in \cite{BaL, GLP} it is solved by an ad hoc construction that has the draw back of using explicitly the dynamics, hence tying the theory to a fixed flow. The problem is also present, but {\bf mistakenly} ignored, in \cite{BuL}. Indeed, Lemma~7.2 of  \cite{BuL} was stated to hold for all $t\geq 0$ but in fact it is only true for all $t$ greater than some non-zero constant. This affects all the paper since \cite{BuL} reduces to the study of the transfer operator, associated with the dynamics, on appropriate Banach  spaces. The norms defining such Banach spaces are based on a set of {\em admissible leaves} that are required to be {\em invariant}, with respect to the dynamics, in the sense of \cite[Lemma 7.2]{BuL}. Since the set of leaves defined in \cite{BuL} is invariant only after some finite time (contrary to the statement of \cite[Lemma 7.2]{BuL}) no control was obtained on the transfer operator for short time. Unfortunately, this problem cannot be fixed as in \cite{BaL}, since in \cite{BuL} one aims at doing perturbation theory and hence one needs a Banach space (on which to study the generator of the flow) which is adapted to an open set of the flows rather than a specific one as in \cite{BaL, GLP}. In section ~\ref{sec:leaves} we show that the present results yield an easy and elegant fix to the above mistake. After such a correction all the results in \cite{BuL} remain true as claimed. Finally, in Appendix \ref{app:strict} we substantiate a claim of \cite{BuL} that, although irrelevant for the results proved there, might have an interest in its own.

\section{Flows on Fibre Bundles}\label{sec:theory}
Let $(E,B,\pi,F)$ be a fibre bundle\footnote{As usual $E$ denotes the total space, $B$ denotes the base space, $\pi: E\to B$ is a continuous surjection, and $F$ denotes the fibre.}
 where $(E, \operatorname{d}_{E})$ and  $(B, \operatorname{d}_{B})$ are both metric spaces and $\pi$ is Lipschitz.\footnote{\label{ft:pi}If one has the natural choice of metric in the base space defined by $\db{p,p'}= \inf\{ \de{x,x'} : p=\pi x, p'=\pi x'\}$ then the Lipschitz property of $\pi$ is immediate.}
 We also assume that there exists an open subset $E_{0} \subset E$ with compact closure and a global\footnote{A section is a continuous map  such that $\pi \circ s(p) = p$ for all $p\in B$. Fibre bundles do not in general have global sections.} section $s: B \to E_{0}\subset E$. The section $s$ is assumed to be Lipschitz. We require that $E_{0}$ contains some neighbourhood\footnote{I.e. $\{x\in E: \de{x, s(B)}<\delta\} \subset E_{0}$ for some $\delta>0$.}  of the set $s(B)$.
Let 
\[
\Psi^{t}: E \to E
\]
 be a bundle flow. This is a jointly continuous\footnote{As a map $  E \times [0,\infty) \to E$.} map $(x,t) \mapsto \Psi^{t}(x) $ such that $\Psi^{0} = \operatorname{Id}$, $\Psi^{s}\circ\Psi^{t}=\Psi^{t+s}$, for each $t$ the map $\Psi^{t}:E \to E$ is Lipschitz, and  there exists  $\varphi^{t}:B\to B$ such that $\varphi^{t}\circ \pi = \pi \circ \Psi^{t}$ for all $t\geq 0$. This means that the flow preserves fibres. We require the set $E_0$ to be eventually invariant, i.e. there exists $t_0>0$ such that $\Psi^{t_{0}} E_{0} \subseteq E_{0}$.
 We also assume that there exists $C>0$ such that, for all $t>0$,\footnote{ Note that, by the semigroup property, the invariance of $E_0$ and the triangle inequality, it suffices that \eqref{eq:Cflow} holds for $t\in[0,t_0]$.}
 \begin{equation}\label{eq:Cflow}
 \sup_{p\in E_0}\de{p,\Psi^{t}p}\leq Ct.
 \end{equation}
Finally, we are interested in the case when  $\Psi^{t}$ is a fibre contraction in the sense that there exists $\sigma <1$, such that 
\begin{equation}\label{eq:contraction}
\de{\Psi^{t_{0}}x,\Psi^{t_{0}}x'} \leq \sigma \de{x,x'}, \quad \text{for all $x,x' \in  E_{0}$, $\pi x = \pi x'$}. 
\end{equation}
As we are interested in robust properties, we wish to consider an open set of flows.
We say that two bundle flows $\Psi$, $\tilde \Psi$ are $\eta$-close if 
\begin{equation}\label{eq:perturb}
\de{\Psi^{t} x, \tilde \Psi^{t} x} \leq \eta t
\quad \text{ for all $x\in E_0$, $t\in[0,t_0]$}.
\end{equation}

The purpose of this section is to prove the following result concerning the construction of a robust strictly-invariant set.
\begin{thm}\label{thm:theory}
Suppose that $\Psi^{t}: E \to E$ is a fibre contracting bundle flow as above. Then there exists a set $K\subset E_0$, and $\eta>0$ such that $\tilde\Psi^{t}K \subseteq K$ for all $t\geq 0$, and  for all $\tilde \Psi$ which are $\eta$-close to $\Psi$.
\end{thm}
\noindent
The remainder of this section is devoted to the proof of the above Theorem.

First note that the fibre contraction \eqref{eq:contraction} and the pre-compactness of $E_0$ implies that there exists $\lambda>0$, $C\geq 1$ such that
 \begin{equation}\label{eq:contraction2}
  \de{  \Psi^t x, \Psi^t x'   } \leq C e^{-\lambda t} \de{x, x'}
  \quad \quad \text{  for all $x,x' \in E_{0}$, $\pi x =\pi x'$, $t\geq 0$}.
 \end{equation}
We would like the bundle map to be a strict fibre contraction. We therefore introduce an adapted metric. Fix some $  \lambda' \in (0, \lambda)$, $t_{1}>0$ sufficiently large so that $Ce^{-(\lambda-\lambda')t_{1}}<1$ and let
\begin{equation}\label{eq:new-metric}
\dep{x,x'} = \int_{0}^{t_{1}} e^{\lambda's} \de{ \Psi^{s} x, \Psi^{s} x'} \ ds.
\end{equation}
\begin{lem}\label{lem:strictcontraction}
 $
  \dep{  \Psi^t x, \Psi^t x'   } \leq e^{-\lambda' t} \dep{x, x'}$
  for all $t\geq 0$, and $x,x' \in E_{0}$,    $\pi x = \pi x'$.
\end{lem}
\begin{proof} 
It suffices to prove the Lemma for $t\in [0,t_1]$. Using the definition of the adapted metric and the estimate \eqref{eq:contraction2} we have
\[
 \begin{aligned}
  \dep{  \Psi^t x, \Psi^t x'   } 
    & = \int_{t}^{t+t_{1}} e^{\lambda' (s-t)} \de{ \Psi^{s} x, \Psi^{s} x'} \ ds\\
    & = e^{-\lambda't} \dep{x, x'} \\
    & \quad +   \int_{0}^{t} e^{\lambda' (s-t)}\left[ e^{\lambda' t_1} \de{ \Psi^{s+t_1} x, \Psi^{s+t_1} x'} -  \de{ \Psi^{s} x, \Psi^{s} x'}\right] \ ds\\
    & \leq e^{-\lambda't} \dep{x, x'} \\
    & \quad + \left[ e^{\lambda' t_1} C e^{-\lambda t_1}   -1 \right]   \int_{0}^{t} e^{\lambda' (s-t)}  \de{ \Psi^{s} x, \Psi^{s} x'}\ ds.
 \end{aligned}
\]
Since $Ce^{-(\lambda-\lambda')t_{1}}<1$ this completes the proof of the lemma.
\end{proof}

\begin{lem}\label{lem:equivalent}
The metrics $\de{\cdot,\cdot}$ and $\dep{\cdot,\cdot}$ are equivalent.
\end{lem}
\begin{proof}
The continuity of the flow, combined with the semigroup properties implies that there exists $C>0$ such that $\de{\Psi^{t}x,\Psi^{t}x'} \leq C \de{x,x'} $ for all $x,x' \in E$, $t\in [-t_{1},t_{1}]$.
This estimate, along with the definition of $\dep{\cdot,\cdot}$ suffices.
\end{proof}

Note that the invariance of $E_0$ and the fibre contraction imply the existence of an invariant global section. Yet, such a section is in general only continuous (e.g. see \cite[Theorem ($3.1'$)]{HPS}). Typically one expects to be able to produce smooth global sections by a smoothing of the invariant one, yet in the present generality it is not obvious how to do it. Hence our assumption on the existence of a global Lipschitz section $s: B \to E_{0}$.

Nevertheless, in the following we  need to have a section close to the invariant one, while still being Lipschitz. This can be easily constructed.
Let   $s_{0} = \Psi^{t_{2}}\circ s \circ \varphi^{-t_{2}} :  B \to E$ for some $t_{2}>0$ chosen sufficiently large as specified shortly.  The section $s_0$ is almost invariant as follows.
\begin{lem}\label{lem:nearinvariant}
For all $\epsilon>0$ there exists $t_2>0$,  and $C_{\epsilon}>0$, such that
\[
\dep{s_{0}\circ\varphi^{t}(p), \Psi^{t}\circ s_{0}(p)} \leq  \epsilon t
\quad \quad
\text{for all $t>0$, $p\in B$}.
\] 
and
$\dep{s_{0}(p),s_{0}(p')} \leq C_{\epsilon} \operatorname{d}_{B}(p,p')$
for all $p,p'\in B$.
\end{lem}
\begin{proof}
Let $t_2 t_0^{-1}\in\bN$ and set $p' =  \varphi^{-t_{2}} (p)$. Using the estimate of Lemma~\ref{lem:strictcontraction} 
\[
\begin{aligned}
\dep{s_{0}\circ\varphi^{t}(p), \Psi^{t}\circ s_{0}(p)} 
& = \dep{ \Psi^{t_{2}}[ s \circ\varphi^{t}(p')], \Psi^{t_{2}}[  \Psi^{t}\circ s (p')]}\\
& \leq e^{-\lambda' t_{2}} \dep{s \circ\varphi^{t}(p'),   \Psi^{t}\circ s (p')}.
\end{aligned}
\]
The linear growth of the flow \eqref{eq:Cflow} and the Lipschitz continuity of $\pi$ means that we have a similar linear growth estimate for $\varphi^{t}:B\to B$. I.e. $\db{x, \varphi^{t}x} \leq Ct$ for all $t>0$ and $x \in B$. 
Now note that 
\[
 \de{s \circ\varphi^{t}(p),   \Psi^{t}\circ s (p)} \leq  \de{s (p),   \Psi^{t}\circ s (p)} +  \de{s \circ\varphi^{t}(p),  s (p)}.
\]
Using the linear growth and the Lipschitz continuity of $s$ we therefore have that 
\[
 \de{s \circ\varphi^{t}(p),   \Psi^{t}\circ s (p)}\leq C t.
 \]
The equivalence of the metrics (Lemma~\ref{lem:equivalent}) means that a similar estimate holds for $\dep{\cdot,\cdot}$.
Choosing $t_{2}$ sufficiently large we ensure that $Ce^{-\lambda' t_{2}}  \leq \epsilon$, which proves the first inequality of the Lemma.

Since, by hypotheses, $B$ is compact and $\varphi^t$  Lipschitz, it follows by the semigroup property and the linear growth  that there exists $C_1>0$ such that $\db{\varphi^t(p),\varphi^t(p')}\leq e^{C_1t}\db{p,p'}$.
Using the latter inequality and Lemma \ref{lem:strictcontraction} the final statement of the lemma follows. Note that it is to be expected that $C_{\epsilon}$ becomes larger as $\epsilon$ is chosen smaller.
\end{proof}

We are now in a position to define the strictly-invariant set.
Let
\[
K = \left\{  x \in E : \dep{x, s_{0}( \pi x)} \leq \tau   \right\}
\]
for some $\tau>0$ which is chosen sufficiently small so that $K \subset E_{0}$.\footnote{ This is possible since $s_0(B)\subset E_0$, by construction, and is a compact set.}

\begin{proof}[Proof of Theorem~\ref{thm:theory}]
We must show that $x\in K$ implies that $\tilde\Psi^{t}x \in K$. By the semigroup property of the flow it suffices to prove the result for  all $t\in [0,t_{3}]$ for some fixed, possibly small, $t_3>0$. 
We estimate
\[
\begin{aligned}
\dep{\tilde\Psi^{t}x, s_{0}\circ \pi( \tilde\Psi^{t}x)}
& \leq         \dep{\tilde\Psi^{t}x, \Psi^{t}x}
		+\dep{ \Psi^{t}x, \Psi^{t}\circ s_{0}\circ \pi( x)}\\
& \ \	+\dep{ \Psi^{t}\circ s_{0}\circ \pi( x), s_{0}\circ \varphi^{t} \circ \pi (x)} \\
& \ \ 	+\dep{s_{0}\circ  \pi (\Psi^{t} x), s_{0}\circ \pi( \tilde\Psi^{t}x)} \\
& \leq C\eta t + e^{-\lambda' t} \tau + t \epsilon + t C_{\epsilon}\eta. 
\end{aligned}
\]
In the last line we used \eqref{eq:perturb} and Lemma~\ref{lem:equivalent}, Lemma~\ref{lem:strictcontraction}, and Lemma~\ref{lem:nearinvariant} for the first three terms respectively. The final term is a combination of \eqref{eq:perturb}, Lemma~\ref{lem:nearinvariant}, the Lipschitz continuity of $\pi$, and \eqref{eq:perturb}.
We now choose $\epsilon =  \lambda' \tau /3$ (this means that $C_{\epsilon}$ is now possibly quite large). Now choose $\eta>0$ sufficiently small so that $(C_{\epsilon} +C)\eta \leq \lambda' \tau/3$. This means that $\dep{\tilde\Psi^{t}x, s_{0}\circ \pi( \tilde\Psi^{t}x)} < \tau$ for $t> 0$ sufficiently small which means that $\tilde\Psi^{t}x \in K$ as required. 
\end{proof}

\section{Robustly Invariant sets in Jet Spaces}\label{sec:application}

Here we apply the abstract results of the previous section to the setting of a $\cC^{r+1}$ Anosov flow
$T_{t} : \cM \to \cM$ for some Riemannian manifold $\cM$. In other words the flow must satisfy the following condition.
\begin{cond}[Anosov Flow]\label{anosov}
At each point $x\in \cM$ there exists a splitting of tangent space $T_x\cM
= E^s(x) \oplus E^f(x) \oplus E^u(x)$, $x\in\cM$. The splitting is
continuous and invariant with
respect to $T_t$. $E^f$ is one dimensional and coincides with the flow
direction. In addition, there exists $C, \lambda>0$ such that
\begin{equation}\label{eq:anosov}
\begin{aligned}
\norm{DT_t v} \leq C e^{- \lambda t} \norm{v} \ \mbox{for each
  $v \in E^s$ and $t \geq 0$},\\
\norm{DT_{-t} v} \leq C e^{- \lambda t} \norm{v} \ \mbox{for each
  $v \in E^u$ and $t \geq 0$},\\
  C^{-1}\norm{v}\leq \norm{DT_{t} v}\leq C\norm{v}\ \mbox{for each
  $v \in E^f$ and $t \in\bR$} .
\end{aligned}
\end{equation}
\end{cond}

First we introduce jet spaces. Given two $d_{s}$-dimensional hypersurfaces containing $p\in \cM$ we say that they are {\em equivalent} if they have an $r$\textsuperscript{th}-order contact at $p$.
For each $p$ the space of $r+1$-jets is denoted $\cJ_{p}$ and is defined to be the equivalence classes of all $d_{s}$-dimensional hypersurfaces containing $p$.  The hypersurfaces  are said to \emph{represent} the jet. 
The jet bundle is the set $\cJ:= \{(p, j):p\in \cM, j \in \cJ_{p}\}$. The space $\cJ$ is a metric space with the metric $\dr{\cdot,\cdot}$ defined as the inf of the $\cC^{r+1}$ distance between hypersurfaces which represent the jets.
This depends on the choice of coordinate charts and so we fix, once and for all, the choice of coordinate charts and use a partition of unity to insure that the metric varies smoothly from point to point on $\cM$. In fact, it is convenient to have co-ordinates charts in which all the relevant hypersufaces can be represented as graphs, then the $\cC^r$ distance would be simply the $\cC^r$ norm of a function. See \cite[Section 3]{BuL} for an explicit choice of such co-ordinates.

Note that every foliation of some region of $\cM$ into $\cC^{r+1}$ $d_{s}$-dimensional hypersurfaces gives rise to a local section in $\cJ$ but not every section can be realised by a foliation of hypersurfaces. We have so defined our fibre bundle.
 There is an induced flow on the jet bundle $\Psi^{t}: \cJ \to \cJ$ 
which we write as 
\[
\Psi^{t}:(p, j) \mapsto \left(T_{-t} (p), \psi_p^t j\right). 
\]
If one considers a hypersurface through the point $p\in \cM$ which is a representation of a jet $j$ then  $\psi_p^t j$ is simply the jet which is represented by the image, under $T_{-t}$, of that hypersurface.
Notice that we are considering the flow in backwards time because we are ultimately interested in almost-stable objects. 

 By standard hyperbolic theory there exists a stable conefield (in tangent space) which we denote by $\bar \cK_p$ at each $p\in \cM$ and which is  invariant under $D_pT_{-t_{0}}$ for some $t_{0}>0$.
We can naturally view $\bar \cK_p$ as a subset $\cK_p$ of $\cJ_p$.  
Hypersurfaces which represent the elements of $\cK_p$ are the hypersurfaces for which the tangent space at $p$ belongs to $\cK_p$. Let $\cK = \{ (p,j): p\in \cM, j \in \cK_{p}  \} \subset \cJ$. It is a standard result of hyperbolic theory that there exists $\sigma \in (0,1)$  such that, perhaps increasing $t_{0}$ if required,\footnote{ For example, it follows from the last two formulae of \cite[Appendix 1]{BuL}.}
\begin{equation}\label{eq:surfaces}
\dr{\Psi^{t_{0}}j,\Psi^{t_{0}}j'} \leq \sigma \dr{j,j'}, \quad \text{for all $j,j' \in \cK_{p}$, $p\in \cM$}. 
\end{equation}

We now wish to apply the results of Section~\ref{sec:theory} to this setting and so we choose $(\cJ, \operatorname{d}_{r})$ as $(E,\operatorname{d}_{E})$, $\cK$ as $E_{0}$, $\cM$ as $B$, and $\operatorname{d}_{B}$ as explained in Footnote~\ref{ft:pi}. 
The existence of a Lipschitz section ($s:\cM \to \cJ$) is easily achieved by taking the section which corresponds to the invariant stable foliation and then smoothing it. 
That the vector field associated to the Anosov flow has bounded $\cC^{r+1}$ norm implies that \eqref{eq:Cflow} is satisfied. The estimate \eqref{eq:surfaces} implies that \eqref{eq:contraction} is satisfied.
Suppose that $\tilde T_{t}:\cM \to \cM$ is another flow and that $V$ and $\tilde V$ are the corresponding vector fields of the two flows. Suppose furthermore that the $\cC^{r+1}$ distance between the two vector fields is less than $\zeta$ for some $\zeta>0$. Then there exists $t_1>0$,  $C>0$ such that
 \[
  \dr{\Psi^t j, \tilde\Psi^t j} \leq C\zeta t
  \quad \text{for all $t\in [0,t_1]$, $j \in \cK$},
  \]
 where $\tilde\Psi^t$ is the bundle flow  corresponding to  $\tilde T_{t}$. This means that \eqref{eq:perturb} is satisfied. 
 All the above means that Theorem~\ref{thm:theory} implies the following.
 \begin{thm}\label{thm:invariance}
 Suppose that $T_{t} : \cM \to \cM$ is a $\cC^{r+1}$ Anosov flow and $\cJ$ denotes the jet bundle associated to $d_{s}$-dimensional submanifolds of $\cM$ as described above. Then there exists $\cJ_{0} \subset \cJ$ and $\eta>0$ such that $\tilde \Psi^t \cJ_{0} \subset \cJ_{0}$ for all $t\geq 0$ and jet bundle flows $\tilde \Psi^t$  associated to a flow which is $\eta$-close to $T_{t}$.
 \end{thm}

\section{Errata Corrige to ``Smooth Anosov Flows: Correlation Spectra and Stability"}\label{sec:leaves}
In this section we show that the above theory provides a nice fix to the error in~\cite{BuL} explained in the introduction. 

Before proceeding, let us remark that the problem in \cite{BuL} can be eliminated in several ways. As already mentioned, a simple solution (used in \cite{BaL, GLP}) is to keep the definition of leaves as in \cite{BuL} and use the dynamics explicitly in the definition of the norm (by taking the new norm to be the sup of the old ones over some time interval). Unfortunately, this is not suited for the task at hand as we are interested also in perturbations of a given dynamics, hence we do not wish to have a norm tied  too closely to the dynamic. One could probably fix the latter issue by taking the sup not only on time but also on a neighbourhood of time dependent dynamics. Yet, this would make the definition of the space rather cumbersome. Instead, we choose to redefine the set of admissible leaves by using the result in section \ref{sec:application}. This has the merit of better elucidating  the geometric properties of the dynamics and could be useful in related problems. Regrettably, in so doing we lose the compact embedding between the Banach spaces \cite[Lemma~2.2]{BuL}, which previously held thanks to \cite[Lemma~2.1]{GL}. The latter does not apply given the new definition of the leaves since it uses in a fundamental way the fact that, in charts, the tangent space of the leaves all belong to the same fixed cone. Nevertheless, it is easy to recover the needed compactness at the level of resolvent operators and hence obtain a correct proof of the results in \cite{BuL}.

Before starting let us note that \cite[Condition 2]{BuL} requires a strict contraction and expansion for all times for stable and unstable tangent vectors. This is used in the (incorrect) construction of the set of invariant leaves. In fact, what  is actually needed is strict expansion in backward time for vectors belonging to the tangent space of the leaves. It is always possible to change the norm in such a way that \cite[Condition 2]{BuL} is satisfied but the strategy suggested to do so in \cite[Footnote 1]{BuL} is too na\"ive. At most, eliminating the absolute value in the exponent, one obtains a norm that contracts strictly along the stable direction. A  norm with the claimed behaviour can be constructed as in  \cite{mather}, but then it has poor smoothness properties. For the reader's convenience and for completeness we show how to construct a smooth strictly-hyperbolic norm in Appendix \ref{app:strict} (at the price of not having the optimal contraction rate nor the orthogonality between invariant subspaces). Here, we content ourselves with a norm that has a strict backward expansion. This suffices for our present needs. Given the new construction of the set of admissible leaves provided shortly, the results of \cite{BuL} holds under  Condition \ref{anosov}, without worrying about adapted norms.

\subsection{ A strictly expanding metric in tangent space}
Let, as in \cite{BuL}, $\cT_\eta$ be the set of flows $\tilde T_t$ defined by vector fields closer than $\eta$ (in the $\cC^{r+1}$ topology) to the vector field of $T_t$.
By Condition \ref{anosov} we have that, for each $\lambda'\in (0,\lambda)$, there exists $\eta_0>0$ such that there exists $C_1>0$ such that, for all $\tilde T_t\in\cT_{\eta_0}$, we have
\[
\norm{\smash{D\tilde T_{-t} v}}\geq C_1e^{\lambda' t} \norm{v}\quad\textrm{ for all $t\geq 0$, $p\in\cM$ and $v\in \bar\cK_p$}.
\]
In analogy with \eqref{eq:new-metric} (and consistently with the aforementioned revision of \cite[Footnote 1]{BuL}) we can define, for some $\lambda''\in (0,\lambda')$,
\[
\norm{v}'=\int_0^{t_4}e^{-\lambda'' s}\norm{DT_{-s}v}ds.
\]
\begin{lem}\label{lem:expansion} There exists $t_4>0$ such that, for all $p\in \cM$, $v\in \bar\cK_p$, $\tilde T_t\in\cT_{\eta_0}$ and $t>0$, we have
\[
\norm{\smash{D\tilde T_{-t}v}}'\geq e^{\lambda''t}\norm{v}'.
\]
\end{lem}
\begin{proof}
The proof is identical to the proof of Lemma \ref{lem:strictcontraction} (apart from the inverse sense of the inequality) and holds provided $C_1e^{(\lambda'-\lambda'')t_4}>1$.
\end{proof}
From now on we will use exclusively such a norm and we will suppress the prime to ease notation.

\subsection{ The set of admissible leaves}
We use Theorem \ref{thm:invariance} for $r+1$-jets. Note that, by construction, the jets in $\cJ_0$ at each $p$ are represented by manifolds whose tangent space at $p$ belongs to $\bar\cK_p$.

We call a $d_s$-dimensional $\cC^{r+1}$ manifold $W$ {\em pre-admissible} if for each point $p\in W$ there exists a neighbourhood of $p$ such that the restriction of $W$ to such a neighbourhood represents a $r+1$-jet in $\cJ_0$. Let $\Sigma_0$ be the set of pre-admissible manifolds. Note that each manifold $W\in\Sigma_0$ has a natural Riemannian structure induced by the one of $\cM$. We can, and will, then talk about balls in $W$ determined by such an induced metric.
Given a fixed $\delta\in (0,1)$ and $R>1$, to be chosen shortly, we call the manifold $W$ {\em admissible} if
\begin{enumerate}
\item $W\in\Sigma_0$;
\item $W$ contains a ball of size $\delta$;
\item $W$ has diameter smaller than $R\delta$. 
\end{enumerate} 
Let $\Sigma$ be the collection of admissible manifolds (leaves).
Here we come to the basic result of this errata: to restrict the set of allowed manifolds so that it has the wanted invariance also for small times and so that the property persists under perturbations.
\begin{lem}
\label{lem:partunity}
There exists $\eta_0>0$ such that, for each $\tilde T_{t}\in\cT_{\eta_0}$, leaf $W\in\Sigma$ and $t\in\bR^+$, there exist
leaves $W_1,\dots, W_\ell\in\Sigma$, whose number $\ell$ is bounded by a
constant depending only on $t$, such that
\begin{enumerate}
\item $\tilde T_{-t}(W) \subset \bigcup_{j=1}^\ell W_j$.
\item $\tilde T_{-t}(W^{+}) \supset \bigcup_{j=1}^\ell W_j^+$.
\item There exists a constant $C$ (independent of $W$ and $t$) such that each
point of $\tilde T_{-t}W^{+}$ is contained in at most $C$ sets $W_j$.
\item There exist functions $\rho_1,\ldots, \rho_\ell$ of class
$\cC^{r+1}$, $\rho_j$ compactly supported on $W_j$, such that $\sum \rho_j=1$
on $\tilde T_{-t}(W)$, and $| \rho_j|_{\cC^{r+1}} \leq C$.
\end{enumerate}
\end{lem}
\begin{proof}
Theorem \ref{thm:invariance} immediately implies that the $D\tilde T_{-t}\Sigma_0\subset\Sigma_0$. If $W\in \Sigma$, then Lemma \ref{lem:expansion} implies that $DT_{-t}W$ satisfies the second requirement of an admissible leaf for all $t>0$. Unfortunately, $DT_{-t}W$ it may grow too much and fail to satisfy the third condition. In such a case note that the manifolds $W'\in T_{-t}\Sigma$ have a uniformly bounded curvature and, by choosing $\delta$ small enough, we can assume that each one of them belongs to some chart of $\cM$. Then there exists $K>0$ such that the distance between two point on $W'$ are less than $K$ times, and more than $K^{-1}$ times, the Euclidean distance in the chart. If $p$ is a point of the manifold at a distance larger than $\delta$ from the boundary, we can consider, in the chart, an  (Euclidean)-ball of centre $p$ and radius $K\delta$. Let us call $W''$ the intersections of the manifold with such a ball. By construction $W''$ will contain a ball (in the induced Reimannian metric) of size at least $\delta$ (hence satisfying property (2)). On the other hand, the diameter of $W''$ will be less than $K^2\delta$, hence satisfying property (3), provided we have chosen $R\geq K^2$.  The Lemma follows then by an application of \cite[Theorem 1.4.10]{Hor}.
\end{proof}
\subsection{ New Banach spaces and old proofs}\label{sec:new-old}

Let us call $\tilde\cB^{p,q}$ the Banach spaces defined in \cite{BuL} (there called simply $\cB^{p,q}$).
We define new norms exactly as before \cite[(2.3)]{BuL}, but with the set $\Sigma$ of admissible leaves as defined above. The new Banach spaces, that we call $\cB^{p,q}$, are then defined again as the closure of the smooth functions in such norms.

Note that new set $\Sigma$ is contained in the old one (possibly with different parameters), moreover now the conclusions of \cite[Lemma~7.2]{BuL} hold true (it is Lemma \ref{lem:partunity} of the previous section). This means that the proofs of \cite[Lemmata~4.1, 4.2, 4.3]{BuL} hold verbatim. Indeed the only problem in the published proofs was the lack of invariance of $\Sigma$ for small times. 

Let us discuss in detail the situation. The results of \cite{BuL} are contained in \cite[Section 5, 6]{BuL}. The arguments in such sections rest on  \cite[Section 7, 8]{BuL} which, in turn, rest on \cite[Lemma 2.2, Sections 3, 4]{BuL}. We have explained above how to modify the definition of the Banach space contained in \cite[Section 3]{BuL}  while \cite[Lemma 2.2]{BuL} is used exclusively in \cite[Sections  4]{BuL}. In conclusion, \cite{BuL} is correct, provided \cite[Section 4]{BuL} is correct, that is if \cite[Lemmata 4.4, 4.5]{BuL} are correct. Since \cite[Lemma 4.5]{BuL} is a direct consequence of \cite[Lemma 4.4]{BuL}, we are left with the problem of checking the latter. The proof of \cite[Lemma 4.4]{BuL} rests on \cite[Lemma 2.2]{BuL} and, in turn, on \cite[Lemma 2.1]{GL}. 
Unfortunately, even though the space is only slightly different form the one in \cite{GL}, as far as we see \cite[Lemma 2.1]{GL} could be false in the present context. Indeed, the proof in \cite{GL} uses in a crucial way that in each chart the cone field is constant, a property that we no longer have. To overcome this obstacle we bypass \cite[Lemma 2.2]{BuL} and provide a direct proof of \cite[Lemma 4.4]{BuL} following the strategy in \cite{Li}.  This provides a complete proof of the results in\cite{BuL}.

\subsection{ Quasi-compactness of the resolvent}
The following Lemma takes the place of \cite[Lemma 4.4]{BuL} and immediately implies \cite[Lemma 4.5]{BuL}.
\begin{lem}
For each $p\in\bN$, $q\in \bR_+$, $q+p<r$, and $z\in\bC$, $\Re(z)=a>0$ the operator $R(z):\cB^{p,q}\to \cB^{p,q}$ has spectral radius bounded by $a^{-1}$ and essential spectral radius bounded by $(a+\bar p\lambda)^{-1}$.
\end{lem}
\begin{proof}
Since the conclusions of \cite[Lemmata~4.1, 4.3]{BuL} hold, the estimate of the spectral radius is as before. It remains to prove the bound on the essential spectral radius.
First of all notice that there exists $K>0$ such that $\cL_t \in L(\cB^{p,q},\tilde\cB^{p,q})$ and $\cL_t \in L(\tilde\cB^{p-1,q+1},\cB^{p-1,q+1})$ for all $t\geq K$.\footnote{ This follows since the set of leaves in \cite{BuL} converge to the stable leaves under the dynamics, hence after some time they will belong to the relevant set.}  From this it follows that  the operators
\[
\begin{split}
R_{K,m}(z)&=\frac 1{(m-1)!}\int_{3K}^\infty t^{m-1}e^{-zt}\cL_t dt\\
&=\cL_K\left[\frac {e^{-2Kz}}{(m-1)!}\int_{K}^\infty (t+2K)^{m-1}e^{-zt}\cL_t dt\right]\cL_K
\end{split}
\]
are compact, as operators in $L(\cB^{p,q},\cB^{p-1,q+1})$. Indeed,  the incorrect proof of the compactness of $R(z)$ in \cite{BuL} holds correct for the operator in square brackets, since no small times are involved, yielding compactness as an operator in $L(\tilde \cB^{p,q},\tilde \cB^{p-1,q+1})$. The result then follows  by the above continuity properties of $\cL_K$. Thus, setting
\[
Q_{K,m}(z)=\frac 1{(m-1)!}\int_{0}^{3K} t^{m-1}e^{-zt}\cL_t dt,
\]
we have $R(z)^m=R_{K,m}+Q_{K,m}$, and \cite[Lemma 4.1]{BuL} implies
\[
\|Q_{K,m}(z)\|_{p,q}\leq C_{p,q} \frac{K^m}{m!}.
\]

We can then conclude by the usual Hennion's argument 
\cite{He} based on Nussbaum's formula \cite{Nu}. Let us recall
the argument. Let $B=\{h\in\cB^{p,q}\;:\;\|h\|_{p,q}\leq 1\}$ and $B_{m}=R_{K,m}(z)B$. By the above discussion $B_{m}$ is compact in $\cB^{p-1,q+1}$. Thus, for each $\epsilon>0$ there are
$h_1,\dots,h_{N_\epsilon}\in B_{m}$ such that
$B_{m}\subseteq\bigcup_{i=1}^{N_\epsilon} 
U_\epsilon(h_i)$, where $U_\epsilon(h_i)=\{h\in\cB\;|\;\|h-h_i\|_{p-1,q+1}<\epsilon\}$. 
For $h\in B_{m}\cap U_\epsilon(h_i)$, \cite[Lemma 4.3]{BuL} implies 
\[
\begin{split}
&\|R(z)^{n}(h-h_i)\|_{p,q}\leq \|R(z)^{n-m}R_{K,m}(h-h_i)\|_{p,q}+C_{p,q} a^{-n+m}\frac{K^m}{m!}\|h-h_i\|_{p,q}\\
&=C_{p,q,\lambda'}(a+\bar p\lambda')^{-n+m}a^{-m}\,
                       +C_{p,q,\lambda',a_0}|z|a^{-n+m}\epsilon+C_{p,q} a^{-n+m}\frac{K^m}{m!} .
\end{split}
\]
Choosing $\epsilon=a^n(a+\lambda\bar p)^{-n+1}$ and $m=\delta n$, for $\delta$ small enough, we conclude that, for each $\lambda''\in(0,\lambda)$,
for each $n\in\bN$  the set $R(z)^nB$ can be covered by a finite number of 
$\|\cdot\|_{p,q}$--balls of radius $C(a+\bar p\lambda'')^{-n}$, which implies that the essential spectral radius of $R(z)$ cannot exceed $(a+\bar p\lambda)^{-1}$.
\end{proof}

\appendix
\section{ Strict hyperbolicity}\label{app:strict}
In this appendix we show that if a flow satisfies Condition \ref{anosov}, then there exists an equivalent smooth metric such that  \cite[Condition 2]{BuL} is satisfied. Note that the content of this appendix is not required for any of the results of the previous sections. 

\begin{lem}\label{lem:metric}
For each $\cC^{r+1}$ Anosov flow there exist a $\cC^{r}$ Riemannian metric ${\|\cdot\|}_{1}$, uniformly equivalent to the original Riemannian metric, and $\sigma>0$ such that
\[
\begin{split}
{\|D_xT_t v\|}_{1} &\geq e^{\sigma t}{\|v\|}_{1}\quad \text{for all } t\geq 0, v\in C^u(x)\\
{\|D_xT_{-t} v\|}_{1} &\geq e^{\sigma t}{\|v\|}_{1}\quad \text{for all }  t\geq 0, v\in C^s(x),
 \end{split}
\]
 where $C^{u}(x)$, $C^{s}(x)$ are invariant unstable and stable conefields, respectively.
\end{lem}
\begin{proof}
By Theorem \ref{thm:invariance} (applied both to $T_t$ and $T_{-t}$) we have the existence of invariant cones $\tilde C^u,\tilde C^s$.
Then, by assumption, there exists $\lambda >0$,  $C>0$ such that
\begin{equation}\label{eq:lambdaestimate}
\begin{split}
&{\|D_xT_t v\|} \geq C e^{\lambda t}{\|v\|}\quad \text{for all } t\geq 0, v\in \tilde C^u(x)\\
&{\|D_xT_{-t} v\|} \geq C e^{\lambda t}{\|v\|}\quad \text{for all } t\geq 0, v\in \tilde C^s(x).
\end{split}
\end{equation}
The construction of the norm is based on the parameter $L>0$, chosen such that $Ce^{{2 \lambda } L}> 1$.
We define the new metric
\[
{\langle v, w\rangle}_{1}=\tfrac{1}{2L}\int_{-L}^{L}\langle DT_s v, DT_s w\rangle ds\;;\quad\quad {\|v\|}_{1}=\sqrt{{\langle v, v\rangle}_{1}}.
\]
We consider the case $v\in C^u:=DT_L \tilde C^u$ and will prove   the first inequality of the lemma.
By the semigroup property of the flow it suffices to prove the statement for $t\in [0,L]$.
\begin{equation}
\begin{split}
{\|DT_t v\|}_{1}^2&=\tfrac{1}{2L} \int_{t-L}^{t+L}\|DT_s v\|^2 ds\\
&={\|v\|}_{1}^2+ \tfrac{1}{2L} \int_{0}^{t}\|DT_{s+{L}} v\|^2 ds-  \tfrac{1}{2L} \int_{0}^{t}\|DT_{s-L} v\|^2 ds\\
&\geq {\|v\|}_{1}^2+  \tfrac{1}{2L}  \int_{0}^{t} (C^{2}e^{4\lambda L} - 1   )   \|DT_{s-L} v\|^2 ds.
\end{split}
\end{equation}
where, in the last line, we used \eqref{eq:lambdaestimate} and the definition of $C^u$.
Note that, for each $L>0$ there exists $C_{L}>0$ such that
\[
{C_{L}}^{-1} {\|v\|}_{1}   \leq  \norm{ DT_{t} v } \leq C_{L}  {\|v\|}_{1} 
\quad \text{for all $v$ and for all $t\in [-L,L]$}.
\]
Combining this with the previous estimate we have shown that, for all $t\in [0,L]$,
\begin{equation}\label{eq:theend}
{\|DT_t v\|}_{1}^2
\geq \left( 1 + t \left[ \frac{ C^{2}e^{4\lambda L} - 1   }{2L C_{L}^{2}}  \right] \right) {\| v\|}_{1}^2.
\end{equation}
Since $C^{2}e^{4\lambda L} - 1 >0$ (because $L>0$ was chosen appropriately) this means that there exists $\sigma>0$ such that $ {\|DT_t v\|}_{1} \geq e^{\sigma t}{\|v\|}_{1} $ as required.

For the other inequality we can argue in complete analogy with the above computation but with time reversed. 
\end{proof}

What we have done is to find a norm for which \eqref{eq:lambdaestimate} holds with $C=1$ (with a different $\lambda>0$). This is very similar to what is proven in \cite{mather} for the case of Anosov diffeomorphisms but here there are two (closely related) differences: 1) the metric $\langle\cdot,\cdot\rangle_1$ is $\cC^r$ rather than just H\"older as in \cite{mather}; 2) contrary to \cite{mather} the new distributions are not orthogonal in the new norms. The latter is annoying but inevitable if one wants a smooth norm. 

Note that, in the above proof, $\sigma$ cannot be taken arbitrarily close to $\lambda$, contrary to \cite{mather}. Indeed \eqref{eq:theend} suggests that typically $\sigma$ will be much smaller than $\lambda$. This is the price of the na\"\i vity of our construction and the requirement that the metric be smooth.

\end{document}